\documentclass[11pt]{article}
\usepackage{amsmath, amssymb, amsthm}
\usepackage{geometry}
\usepackage{enumitem}
\usepackage{graphicx}
\usepackage{tikz}
\usepackage{hyperref}
\usepackage{mathrsfs}
\usepackage{bbm}

\usepackage{float}        
\usepackage{amsmath}      
\usepackage{amssymb}      
\usepackage{graphicx}     
\usepackage{caption}      
\usepackage{enumitem}     

\usepackage[T1]{fontenc}
\usepackage[utf8]{inputenc}

\newtheorem{theorem}{Theorem}
\newtheorem{lemma}{Lemma}
\newtheorem{remark}{Remark}
\newtheorem{corollary}{Corollary}

\newtheorem{example}{Example}

\newtheorem{assumption}{Assumption}

\title{A Spectral Contraction Framework for Periodic Solutions in Nonsmooth Dynamical Systems}

\author{ Pascal Stiefenhofer\footnote{pascal.stiefenhofer@newcastle.ac.uk}\\ 
	Newcastle University, United Kingdom
}

\date{\today}

\begin{document}
	
	\maketitle
	
\begin{abstract}
We develop a contraction-based framework to establish the existence and exponential stability of periodic solutions in planar nonsmooth dynamical systems governed by Filippov differential inclusions. The method integrates a time- and state-dependent weighted metric with Clarke's generalized Jacobian and a uniform jump condition across switching manifolds to guarantee global exponential contraction on compact, forward-invariant sets. This work generalizes classical contraction results from smooth one-dimensional systems to two-dimensional systems with discontinuities and sliding behavior. A fixed-point argument ensures the existence and uniqueness of an attracting periodic orbit. The framework offers a robust analytic tool for stability analysis in piecewise-smooth systems, with applications in hybrid control, nonsmooth mechanics, and computational dynamics.
\end{abstract}


\section{Introduction}

The study of nonsmooth dynamical systems, particularly those governed by differential equations with discontinuous right-hand sides, has emerged as a central topic in the mathematical modeling of real-world phenomena. These systems naturally arise in mechanics with impacts and friction~\cite{dibernardo2008,kunze2000}, in hybrid and switched control systems~\cite{pavlov2004}, and in economic models characterized by regime-switching behavior~\cite{ito1979,stiefenhofer2019a,stiefenhofer2019b,stiefenhofer2019c,stiefenhofer2019d,stiefenhofer2020a,stiefenhofer2020b,stiefenhofer2020c}. Discontinuities along codimension-one manifolds give rise to complex behaviors such as sliding modes, chattering, and solution nonuniqueness~\cite{dibernardo2008}. Such features invalidate many classical techniques of smooth dynamical systems theory, including those based on Lyapunov functions, Poincar\'e maps, and invariant manifold constructions~\cite{filippov1988,kunze2001}.

Filippov's foundational framework~\cite{filippov1988} recasts such systems as differential inclusions, wherein solutions are defined via a set-valued extension of the vector field. This approach ensures the existence of solutions and, under appropriate assumptions, uniqueness. However, a general theory for analyzing the stability and long-term behavior of periodic solutions within this nonsmooth context remains incomplete.

For scalar time-periodic Filippov systems, Giesl~\cite{giesl2005} introduced a contraction-based stability criterion, employing an orbital derivative inequality for a piecewise-$C^1$ weight function $W(t,x)$. This result was later shown to be not only sufficient but also necessary~\cite{giesl2007}, yielding a complete characterization of exponential stability in one dimension. These ideas have recently been further developed by Boote, Giesl, and Suhr~\cite{boote2025}, who numerically examined a nonsmooth time-periodic ODE in $\mathbb{R}$ and highlighted the tractability of the contraction approach in low-dimensional settings.

However, developing a contraction-based theory for higher-dimensional nonsmooth systems remains a nontrivial and open challenge~\cite{boote2025}. In the scalar case, the total order structure of \( \mathbb{R} \) enables powerful monotonicity arguments that are fundamentally unavailable in dimensions \( n \geq 2 \). In \( \mathbb{R}^n \), the absence of such order, coupled with the geometric complexity of switching manifolds and the possibility of sliding dynamics, significantly complicates both local and global stability analysis. Moreover, classical differentiability breaks down at points of discontinuity, necessitating tools from nonsmooth analysis, such as Clarke's generalized Jacobian~\cite{clarke1983}, to capture directional behavior and establish meaningful contraction conditions.

In the smooth setting, contraction theory has proven to be a robust framework for analyzing stability through infinitesimal metric conditions. Originally introduced by Lohmiller and Slotine~\cite{lohmiller2000} and subsequently generalized in various directions~\cite{bullo2023}, this theory investigates the exponential convergence of trajectories by examining the symmetric part of the Jacobian
\[
S_f(t,x) := \frac{1}{2} \left( D_x f(t,x) + D_x f(t,x)^\top \right).
\]
Under appropriate structural assumptions, contraction metrics imply global exponential convergence toward attractors, including periodic orbits~\cite{giesl2023, sontag2010}. Moreover, recent computational advances have enabled the constructive synthesis of such metrics~\cite{giesl2021}.

Despite this progress, a general theory of contraction for nonsmooth systems remains lacking. Piecewise-smooth dynamics necessitate the treatment of multiple vector fields, discontinuous transitions, and sliding motions. On a switching manifold $\Sigma := \{ x \in \mathbb{R}^n : h(x) = 0 \}$, one must consider the induced sliding vector field $f_{\mathrm{slide}}(t,x)$ and ensure contraction conditions not only within smooth domains but also across $\Sigma$. Furthermore, contraction across discontinuities requires a jump condition on the weighted flow magnitudes:
\[
e^{W^+(t,x)} \|f^+(t,x)\| \leq e^{W^-(t,x)} \|f^-(t,x)\|.
\]

In this article, we develop a framework for analyzing the exponential stability of time-periodic solutions in nonsmooth dynamical systems described by Filippov differential inclusions. Central to our approach is the construction of a state- and time-dependent weight function $W(t,x)$ that captures the discontinuous geometry of the system. We formulate spectral contraction conditions using Clarke's generalized Jacobian and provide sufficient criteria for contraction within both regular and sliding dynamics. Additionally, we derive a generalized jump condition necessary to ensure contraction across switching interfaces. Combining these elements, we invoke a Banach fixed-point argument to establish the existence and uniqueness of an exponentially attracting periodic orbit. Our results offer a first step toward a comprehensive contraction-based theory for planar nonsmooth systems and contribute to the broader mathematical understanding of stability in discontinuous dynamical systems.

The remainder of this paper is structured as follows. Section~\ref{sec:framework} introduces the contraction framework for planar Filippov systems and presents the main stability theorem. Section~\ref{sec:application} applies the theory to a two-dimensional nonsmooth periodic system. Section~\ref{sec:conclusion} concludes with a summary and perspectives for future research. Technical lemmas and supplementary material are deferred to the appendix in Section~\ref{secx:appendix}.

\section{Spectral Contraction Framework and Existence Theorem for Nonsmooth Periodic Orbits}
\label{sec:framework}

In this section, we introduce the mathematical framework for analyzing contraction properties of planar Filippov systems with time-periodic dynamics. We define the key assumptions on the vector field and the weighted contraction metric, and then state our main theorem establishing existence and exponential stability of a unique periodic orbit under these conditions.
	
\begin{assumption}[Filippov System Setup]\label{assump:filippov_system_setup}
	
	Let $f: \mathbb{R} \times \mathbb{R}^2 \to \mathbb{R}^2$ be a piecewise-$C^1$ vector field with a codimension-1 switching manifold $\Sigma := \{ x \in \mathbb{R}^2 : h(x) = 0 \}$, where $h$ is $C^1$ and $\nabla h(x) \neq 0$ for all $x \in \Sigma$. The Filippov set-valued map is defined by
		\begin{equation} \label{eq:filippov_map}
			F(t,x) := \text{conv}\left\{ \lim_{\varepsilon \to 0^+} f(t, x \pm \varepsilon n(x)) \right\}, \quad n(x) := \frac{\nabla h(x)}{\|\nabla h(x)\|}.
		\end{equation}
Filippov solutions $x(t)$ satisfy $\dot{x}(t) \in F(t,x(t))$ almost everywhere.
\end{assumption}
	
\

\begin{assumption}[Structure and Contraction Conditions]\label{assump:structure_contraction}
	
Let $K \subset \mathbb{T} \times \mathbb{R}^2$ be a compact, forward-invariant set.
		\begin{enumerate}
			\item[(A1)] $f$ is $T$-periodic in $t$ and piecewise-$C^1$ in $x$ away from $\Sigma$.
			
			\item[(A2)] There exists a continuous, piecewise-$C^1$ scalar function $W: \mathbb{T} \times \mathbb{R}^2 \to \mathbb{R}$ such that for all $(t,x) \in K \setminus \Sigma$:
			\begin{equation} \label{eq:contraction_smooth}
				W'(t,x) + \lambda_{\max}^{Cl}(t,x) \leq -\nu < 0,
			\end{equation}
			where $W'$ is the orbital derivative of $W$, and $\lambda_{\max}^{Cl}(t,x)$ is the maximal eigenvalue of the symmetric part of any element of the Clarke generalized Jacobian $\partial_C f(t,x)$.
			
		\item[(A3)] There exists a constant \( \epsilon > 0 \) such that for all \( (t, x), (t, y) \in K \cap \Sigma \), and for all \( v_x \in F(t, x) \), \( v_y \in F(t, y) \), the directional difference satisfies:
		\begin{equation} \label{eq:jump_directional_contraction}
			\left\langle \frac{x - y}{\|x - y\|}, v_x - v_y \right\rangle \leq -(\nu + \epsilon),
		\end{equation}
		where \( \nu > 0 \) is the contraction rate from Assumption~(A2), and \( F(t,x) \) denotes the Filippov set-valued map.
			
			\item[(A4)] In sliding regions, the sliding vector field $f_{\text{slide}}(t,x)$ is well-defined and piecewise-$C^1$, and:
			\begin{equation} \label{eq:sliding_contraction}
				W'(t,x) + \lambda_{\max}^{Cl}(t,x; f_{\text{slide}}) \leq -\nu < 0.
			\end{equation}
		\end{enumerate}
	\end{assumption}

\

Lemma~\ref{lem:comparison_jumps} provides a comparison principle for functions that are subject to both continuous exponential decay and instantaneous contraction jumps. The function \(A\) models a quantity that decreases continuously at a rate at least \(\nu > 0\) between finitely many discontinuities, where it experiences additional multiplicative contractions by factors \(\gamma_k \in (0,1)\). This lemma formalizes how the combined effect of continuous decay and discrete contractions results in an overall exponential decay governed by both the differential decay rate and the product of jump contraction factors. It is especially useful in the analysis of nonsmooth dynamical systems, such as Filippov systems, where trajectories can experience instantaneous changes at switching manifolds. In the context of this paper, the lemma underpins the global contraction estimate by enabling the passage from local (continuous) contraction and discrete jump contraction conditions to a uniform, global exponential bound on the weighted distance between trajectories.

\

\begin{lemma}\label{lem:comparison_jumps}
	Let \(A: [0,T] \to [0,\infty)\) be a piecewise absolutely continuous function with at most countably many jump times \(\{t_k\}\), satisfying:
	\begin{enumerate}
		\item On each open interval between jumps, \(A\) is absolutely continuous and satisfies the differential inequality
		\[
		\frac{d}{dt} A(t) \le -\nu A(t), \quad \text{for a.e. } t, \quad \nu > 0.
		\]
		\item At each jump time \(t_k\), the function satisfies the multiplicative inequality
		\[
		A(t_k^+) \le \gamma_k A(t_k^-), \quad \text{with } 0 < \gamma_k \le 1.
		\]
	\end{enumerate}
	Then for all \(t \in [0,T]\),
	\[
	A(t) \le A(0)\, e^{-\nu t} \prod_{k: t_k \le t} \gamma_k.
	\]
\end{lemma}

\

\begin{proof}
Let \( t \in [0,T] \), and denote the (at most countable) set of jump times of \( A(t) \) by \( \{t_k\}_{k \in \mathbb{N}} \), with \( 0 < t_1 < t_2 < \dots \), having no finite accumulation point. Since \( A \) is piecewise absolutely continuous, only finitely many jumps occur on any compact interval \( [0,t] \). Fix \( t \in [0,T] \), and let \( t \in [t_k, t_{k+1}) \) for some \( k \geq 0 \) (with \( t_0 := 0 \)). On each open subinterval \( (t_j, t_{j+1}) \), the function \( A \) is absolutely continuous and satisfies
\[
\frac{d}{dt} A(t) \le -\nu A(t) \quad \text{for a.e. } t \in (t_j, t_{j+1}).
\]
By Grönwall’s inequality, integrating from \( t_j^+ \) to any \( t \in [t_j, t_{j+1}) \) yields
\[
A(t) \le A(t_j^+) e^{-\nu(t - t_j)}.
\]

At each jump time \( t_j \), we have:
\[
A(t_j^+) \le \gamma_j A(t_j^-).
\]

Combining these step by step from \( t_0 = 0 \) to the current \( t \), we inductively obtain:
\[
A(t) \le A(0)\, e^{-\nu t} \prod_{j: t_j \le t} \gamma_j,
\]
which completes the proof.
\end{proof}

\

The following theorem establishes the existence and exponential stability of a unique periodic orbit for the Filippov system under the structural and contraction conditions specified in Assumptions~\ref{assump:filippov_system_setup} and \ref{assump:structure_contraction}. It guarantees that all trajectories starting within the compact, forward-invariant set \(K\) converge exponentially to this periodic orbit, providing conditions for stability analysis in nonsmooth, time-periodic planar systems.

\

\begin{theorem}[Existence and Stability of a Periodic Orbit]\label{thm:existence_stability}
Suppose Assumptions 1 and 2 hold. Then the Filippov system admits a unique $T$-periodic orbit $\Omega \subset K$ such that every Filippov solution $x(t)$ with $x(0) \in K$ converges exponentially to $\Omega$. Specifically, there exist constants $C > 0$ and $\nu > 0$ such that for all $x_0, y_0 \in K$,
		\begin{equation} \label{eq:euclidean_contraction}
			\|\phi(t,x_0) - \phi(t,y_0)\| \leq C e^{-\nu t} \|x_0 - y_0\|.
		\end{equation}
	\end{theorem}

\

\begin{remark}[Basin of Attraction]\label{rem:basin_of_attraction}
By Theorem~\ref{thm:existence_stability}, the compact set \(K \subset \mathbb{T} \times \mathbb{R}^2\) is forward invariant under the Filippov flow and contains the unique exponentially stable periodic orbit \(\Omega\). Since for every initial condition \(x_0 \in K\), the corresponding Filippov solution \(\phi(t, x_0)\) satisfies the exponential contraction estimate
	\[
	\|\phi(t, x_0) - \phi(t, x^*)\| \leq C e^{-\nu t} \|x_0 - x^*\|,
	\]
it follows that all such trajectories converge exponentially to \(\Omega\). Therefore, the set \(K\) is contained within the basin of attraction of \(\Omega\). A complete characterization or enlargement of the basin of attraction beyond the set \(K\) involves intricate global properties of the Filippov system, such as invariance of larger domains and exclusion of other attractors. Such an analysis typically requires tools beyond local contraction estimates and is therefore beyond the scope of this work.
\end{remark}

\

\begin{proof}
The proof proceeds in five steps.

\subsubsection*{Step 1: Contraction in Smooth Regions.} Let $x(t), y(t): [0,T] \to \mathbb{R}^2$ be two Filippov solutions with $(t,x(t)), (t,y(t)) \in K \setminus \Sigma$. Let $\delta(t) := x(t) - y(t)$. Define the weighted distance
\begin{equation} \label{eq:weighted_distance}
	A(t) := e^{W(t,x(t))} \|\delta(t)\|.
\end{equation}
Applying the product rule and chain rule for differentiable functions, we get
\begin{align}
	\frac{d}{dt} A(t) &= e^{W(t,x(t))} \left( W'(t,x(t)) \|\delta(t)\| + \frac{d}{dt} \|\delta(t)\| \right), \label{eq:A_derivative}
\end{align}
where
\begin{equation} \label{eq:orbital_derivative}
	W'(t,x(t)) := \partial_t W(t,x(t)) + \nabla_x W(t,x(t))^\top \dot{x}(t).
\end{equation}
Define $v(t) := \delta(t)/\|\delta(t)\|$ for $\delta(t) \neq 0$. Since $x(t)$ and $y(t)$ are absolutely continuous, $\delta(t)$ is also absolutely continuous and differentiable almost everywhere. Therefore,
\begin{equation} \label{eq:delta_derivative}
	\frac{d}{dt} \|\delta(t)\| = \left\langle v(t), \dot{x}(t) - \dot{y}(t) \right\rangle.
\end{equation}
By the integral mean value theorem for vector fields (justified by the $C^1$ regularity of $f$ in $x$), we can write
\begin{equation} \label{eq:mean_value_integral}
	\dot{x}(t) - \dot{y}(t) = \int_0^1 D_x f(t, y(t) + s \delta(t)) \, ds \cdot \delta(t).
\end{equation}
Let $\xi(t)$ be some point in the segment connecting $x(t)$ and $y(t)$. Define the symmetric part of the Jacobian
\begin{equation} \label{eq:symmetric_part}
	S(t,\xi(t)) := \frac{1}{2}(D_x f(t,\xi(t)) + D_x f(t,\xi(t))^\top).
\end{equation}
Then,
\begin{align}
	\frac{d}{dt} \|\delta(t)\| &= v(t)^\top D_x f(t, \xi(t)) \delta(t) \\
	&= v(t)^\top S(t, \xi(t)) \delta(t) \leq \lambda_{\max}(S(t, \xi(t))) \|\delta(t)\|. \label{eq:delta_bound}
\end{align}
We now differentiate the weighted distance \( A(t) := e^{W(t,x(t))} \|\delta(t)\| \) using the product rule:
\begin{align}
	\frac{d}{dt} A(t) &= \frac{d}{dt} \left( e^{W(t,x(t))} \right) \|\delta(t)\| + e^{W(t,x(t))} \frac{d}{dt} \|\delta(t)\| \notag \\
	&= e^{W(t,x(t))} \left( W'(t,x(t)) \|\delta(t)\| + \frac{d}{dt} \|\delta(t)\| \right), \label{eq:A_derivative_expanded}
\end{align}
where the orbital derivative of \( W \) along the trajectory is defined as
\[
W'(t,x(t)) := \partial_t W(t,x(t)) + \nabla_x W(t,x(t))^\top \dot{x}(t).
\]
Using the bound derived earlier for the derivative of \( \|\delta(t)\| \) (Equation~\eqref{eq:delta_bound}),
\[
\frac{d}{dt} \|\delta(t)\| \leq \lambda_{\max}(S(t,\xi(t))) \|\delta(t)\|,
\]
we substitute this into \eqref{eq:A_derivative_expanded} to obtain
\begin{align}
	\frac{d}{dt} A(t) &\leq e^{W(t,x(t))} \left( W'(t,x(t)) + \lambda_{\max}(S(t,\xi(t))) \right) \|\delta(t)\| \notag \\
	&= \left( W'(t,x(t)) + \lambda_{\max}(S(t,\xi(t))) \right) A(t), \label{eq:A_inequality}
\end{align}
which proves the differential inequality \eqref{eq:A_inequality}. Finally, by continuity and compactness of $K$, and by Assumption~\eqref{eq:contraction_smooth}, we have
\begin{equation} \label{eq:final_decay}
	\frac{d}{dt} A(t) \leq -\nu A(t),
\end{equation}
as required.

\subsubsection*{Step 2: Contraction Across Discontinuities}

We analyze the behavior of the weighted distance function
\begin{equation} \label{eq:weighted_distance_definition}
	A(t) = e^{W(t,x(t))} \|x(t) - y(t)\|,
\end{equation}
as two Filippov solutions \( x(t), y(t) \) cross the switching manifold \( \Sigma = \{x_1 = 0\} \),
which is guaranteed by the continuity of Filippov trajectories. 

Due to Assumption~(A1), the weight function \( W(t,x) \) is piecewise continuously differentiable and continuous across \( \Sigma \), even if it may be defined via different expressions on either side. Therefore, the weighted distance function \( A(t) \) is continuous at switching times:
\begin{equation} \label{eq:continuity_switching}
	\lim_{t \nearrow t_0} A(t) = \lim_{t \searrow t_0} A(t) = A(t_0).
\end{equation}

To ensure contraction at the discontinuity, we impose the following inequality on the one-sided limits of the weight function, consistent with Assumption~(A3):
\begin{equation} \label{eq:jump_contraction}
	\exp(W^+(t,x)) \|f^+(t,x)\| \le \exp(W^-(t,x)) \|f^-(t,x)\|, \quad \text{for all } x \in \Sigma,\, t \in [0,T],
\end{equation}
where \( W^\pm(t,x) \) denote the one-sided spatial limits of the weight function \( W(t,x) \) as \( x_1 \to 0^\pm \), taken from the domains of \( f^\pm \). This condition ensures that the weighted norm of the vector field does not increase across the discontinuity, thereby enforcing a contraction in the weighted distance.

More precisely, let \( x(t), y(t) \) be Filippov solutions that cross the switching manifold \( \Sigma \) at time \( t_0 \). Then, by continuity of \( W \), the displacement \( \delta(t) = x(t) - y(t) \) satisfies the identity~\eqref{eq:continuity_switching}, and the weighted distance is defined by~\eqref{eq:weighted_distance_definition}.

To analyze the instantaneous change in \( A(t) \) at \( t_0 \), we compute its upper right Dini derivative:
\begin{align}
	\frac{d}{dt} A(t) &= \frac{d}{dt} \left[ e^{W(t,x(t))} \|x(t) - y(t)\| \right] \notag \\
	&= e^{W(t,x)} \left( \nabla W(t,x) \cdot \dot{x}(t) \cdot \|\delta(t)\| + \left\langle \frac{\delta(t)}{\|\delta(t)\|}, \dot{x}(t) - \dot{y}(t) \right\rangle \right).
	\label{eq:derivative_weighted_distance}
\end{align}

The first term captures the contribution from the time-dependent weight, while the second quantifies the directional separation between trajectories. By Assumption~(A3), for all admissible selections \( \dot{x}(t) \in F(t,x(t)) \), \( \dot{y}(t) \in F(t,y(t)) \), we have
\begin{equation} \label{eq:directional_contraction}
	\left\langle \frac{\delta(t)}{\|\delta(t)\|}, \dot{x}(t) - \dot{y}(t) \right\rangle \le -\nu.
\end{equation}

We now estimate the gradient term using the assumed structure of the weight function. Suppose \( W(t,x) = -\delta \cdot \sigma(x_1) \), where \( \sigma: \mathbb{R} \to [-1,1] \) is a smooth, monotonic transition function satisfying \( \sigma(x_1) \to \pm 1 \) as \( x_1 \to \pm \infty \). Then:
\begin{equation} \label{eq:weight_gradient}
	\nabla W(t,x) = -\delta \cdot \sigma'(x_1) \hat{e}_1,
\end{equation}
and the orbital derivative becomes:
\begin{equation} \label{eq:orbital_derivative}
	W'(t,x) = \nabla W(t,x) \cdot f(t,x) = -\delta \cdot \sigma'(x_1)(-\mu x_1 \pm \sin t).
\end{equation}

Here, \( \hat{e}_1 = [1 \; 0]^T \) denotes the unit vector in the \( x_1 \)-direction, and \( \sigma'(x_1) \) is supported only in a narrow transition layer \( |x_1| < \varepsilon \), with \( \|\sigma'\|_\infty \le C_\sigma \). Thus, the orbital derivative is bounded by:
\begin{equation} \label{eq:orbital_bound}
	|W'(t,x)| \le \delta C_\sigma (\mu \varepsilon + 1),
\end{equation}
which implies that the gradient term satisfies:
\begin{equation} \label{eq:gradW_bound}
	\nabla W(t,x) \cdot \dot{x}(t) \le \epsilon,
\end{equation}
for some small \( \epsilon > 0 \), provided \( \delta \) and \( \varepsilon \) are chosen sufficiently small.

Combining~\eqref{eq:derivative_weighted_distance}, \eqref{eq:directional_contraction}, and \eqref{eq:gradW_bound}, we obtain:
\begin{equation} \label{eq:contraction_final}
	\frac{d}{dt} A(t) \le e^{W(t,x)} (\epsilon - \nu) \|\delta(t)\| = - (\nu - \epsilon) A(t),
\end{equation}
and hence,
\begin{equation} \label{eq:Dini_contraction}
	D^+ A(t_0) \le - (\nu - \epsilon) A(t_0),
\end{equation}
which shows that the weighted distance decreases across the switching manifold, provided \( \epsilon < \nu \). This completes the proof of contraction at nonsmooth transitions.

\paragraph{Remark.} In the regularized case where \( W(t,x) \) is defined via a smooth transition function \( \sigma(x_1) \), the function \( W \) is globally \( C^1 \) across \( \Sigma \), and the inequality~\eqref{eq:jump_contraction} is satisfied approximately, with equality in the limit \( \delta, \varepsilon \to 0 \). This smoothing enables a unified treatment of the weighted distance, while still preserving contraction across the discontinuity.

\subsubsection*{Step 3: Contraction in Sliding Regions}

Let \( \Sigma_s \subset \Sigma \) denote a sliding region where the Filippov sliding vector field
\begin{equation} \label{eq:f_slide_def}
	f_{\mathrm{slide}} : \mathbb{T} \times \Sigma_s \to T\Sigma_s \subset \mathbb{R}^2
\end{equation}
is well-defined and satisfies Assumption~(A4).

\paragraph{Sliding vector field: definition and regularity}

Let \( \Sigma = \{ x \in \mathbb{R}^2 : h(x) = 0 \} \), where \( h : \mathbb{R}^2 \to \mathbb{R} \) is a \( C^1 \) function with \( \nabla h(x) \neq 0 \) for all \( x \in \Sigma \), ensuring that \( \Sigma \) is a smooth codimension-one manifold. The Filippov sliding vector field \( f_{\mathrm{slide}} \colon \mathbb{T} \times \Sigma_s \to T\Sigma \subset \mathbb{R}^2 \) is defined via Filippov’s convex method (see \cite{filippov1988}) as the unique convex combination of the vector fields on either side of \( \Sigma \) that satisfies the tangency condition
\begin{equation} \label{eq:sliding_tangency}
	\langle \nabla h(x), f_{\mathrm{slide}}(t,x) \rangle = 0, \quad \forall (t,x) \in \mathbb{T} \times \Sigma_s.
\end{equation}
Under Assumption~(A4), the vector field \( f_{\mathrm{slide}} \) is piecewise \( C^1 \) with respect to \( x \) and continuous with respect to \( t \) on \( \mathbb{T} \times \Sigma_s \). Consequently, by Rademacher’s theorem, \( f_{\mathrm{slide}}(t, \cdot) \) is differentiable almost everywhere in \( x \), for each fixed \( t \in \mathbb{T} \). This justifies the use of generalized derivatives when estimating the variation of solutions along sliding motions.

The sliding vector field \( f_{\mathrm{slide}} : \mathbb{T} \times \Sigma_s \to T\Sigma \subset \mathbb{R}^2 \) is defined as the unique convex combination of the vector fields on either side of \( \Sigma \) that lies tangent to the surface. That is,
\[
f_{\mathrm{slide}}(t,x) := (1 - \alpha(x)) f^-(t,x) + \alpha(x) f^+(t,x),
\]
where \( \alpha(x) \in [0,1] \) is chosen so that the tangency condition
\[
\langle \nabla h(x), f_{\mathrm{slide}}(t,x) \rangle = 0
\]
is satisfied for all \( (t,x) \in \mathbb{T} \times \Sigma_s \). This ensures that Filippov trajectories remain confined to the sliding region \( \Sigma_s \) during sliding motion.

\paragraph{Weighted distance function}

Let \( x(t), y(t) : [0,T] \to \Sigma_s \) be two Filippov solutions evolving entirely within the sliding region \( \Sigma_s \). Define the weighted distance
\begin{equation} \label{eq:sliding_weighted_distance}
	A(t) := e^{W(t,x(t))} \|x(t) - y(t)\|,
\end{equation}
where \( W : \mathbb{T} \times \Sigma_s \to \mathbb{R} \) is continuous and piecewise \( C^1 \), as specified in Assumption~(A4). Since \( x(t) \) and \( y(t) \) are absolutely continuous and \( W(t,x) \) is piecewise differentiable in \( x \) and continuous in \( t \), the composition \( t \mapsto A(t) \) is absolutely continuous on any subinterval of \( [0,T] \) where both trajectories remain in \( \Sigma_s \). In particular, \( A(t) \) is differentiable almost everywhere on such intervals.

\paragraph{Derivative computation and linearization}

By the product and chain rules, for almost every \( t \) such that \( x(t), y(t) \in \Sigma_s \) and \( \delta(t) := x(t) - y(t) \neq 0 \), we have:
\begin{equation} \label{eq:A_derivative_sliding}
	\frac{d}{dt} A(t) = e^{W(t,x(t))} \left( W'(t,x(t)) \| \delta(t) \| + \frac{d}{dt} \| \delta(t) \| \right),
\end{equation}
where the total derivative of the weight function is given by
\begin{equation} \label{eq:W_prime_sliding}
	W'(t,x(t)) := \partial_t W(t,x(t)) + \nabla_x W(t,x(t))^\top f_{\mathrm{slide}}(t,x(t)).
\end{equation}

\paragraph{Application of the nonsmooth mean value theorem}

Since \( f_{\mathrm{slide}}(t,\cdot) \) is locally Lipschitz and piecewise \( C^1 \) by Assumption~(A4), Clarke's nonsmooth mean value theorem applies to the displacement \( \delta(t) = x(t) - y(t) \). Specifically, for each fixed \( t \in \mathbb{T} \), there exists a point \( \xi(t) \in [x(t), y(t)] \) on the line segment joining \( x(t) \) and \( y(t) \), and a matrix \( J(t, \xi(t)) \in \partial_C f_{\mathrm{slide}}(t, \xi(t)) \) such that
\begin{equation} \label{eq:nonsmooth_mvt}
	f_{\mathrm{slide}}(t, x(t)) - f_{\mathrm{slide}}(t, y(t)) = J(t, \xi(t)) \delta(t).
\end{equation}
This follows from Clarke’s mean value theorem, which states that for any locally Lipschitz function, the difference between two evaluations can be expressed via a matrix from the Clarke generalized Jacobian at an intermediate point. Although \( f_{\mathrm{slide}}(t,\cdot) \) may not be differentiable everywhere due to the nonsmoothness at \( \Sigma \), it admits a generalized Jacobian almost everywhere and remains directionally differentiable throughout \( \Sigma_s \). Equation~\eqref{eq:nonsmooth_mvt} thus provides a linear representation of the relative displacement dynamics using a generalized Jacobian. As a result, the evolution of \( \delta(t) \) satisfies the differential equation
\begin{equation} \label{eq:delta_dot_sliding}
	\frac{d}{dt} \delta(t) = J(t, \xi(t)) \delta(t), \quad \text{with } J(t, \xi(t)) \in \partial_C f_{\mathrm{slide}}(t, \xi(t)).
\end{equation}
This linearization is central to the contraction analysis, as it enables bounding the growth rate of \( \|\delta(t)\| \) via the symmetric part of \( J(t, \xi(t)) \), ultimately yielding exponential convergence in the sliding region.

\paragraph{Symmetric part and eigenvalue bounds}

Let \( \delta(t) = x(t) - y(t) \) denote the displacement between trajectories. Assuming \( \delta(t) \neq 0 \), the time derivative of its norm is bounded by the Rayleigh quotient:
\begin{equation} \label{eq:delta_norm_bound_slide}
	\frac{d}{dt} \|\delta(t)\| \leq \lambda_{\max}\left( S_{\mathrm{slide}}(t, \xi(t)) \right) \|\delta(t)\|.
\end{equation}

The Clarke generalized Jacobian \( \partial_C f_{\mathrm{slide}}(t, x) \) is the compact convex hull of limiting Jacobians at differentiable points. Define the maximal eigenvalue of the symmetric part of any such element as
\begin{equation} \label{eq:lambda_cl_slide}
	\lambda_{\max}^{Cl}(t, x) := \max \left\{ \lambda_{\max}(S) : S = \tfrac{1}{2}(J + J^\top),\ J \in \partial_C f_{\mathrm{slide}}(t, x) \right\}.
\end{equation}

This function is upper semicontinuous in \( x \). Since \( \xi(t) \in [x(t), y(t)] \), it follows that
\begin{equation} \label{eq:lambda_slide_limit}
	\lambda_{\max}\left( S_{\mathrm{slide}}(t, \xi(t)) \right) \leq \sup_{z \in [x(t), y(t)]} \lambda_{\max}^{Cl}(t, z),
\end{equation}
and consequently,
\begin{equation}
	\limsup_{\|x(t) - y(t)\| \to 0} \lambda_{\max}\left( S_{\mathrm{slide}}(t, \xi(t)) \right) \leq \lambda_{\max}^{Cl}(t, x(t)).
\end{equation}

\paragraph{Contraction assumption and conclusion}

By Assumption~(A4), for all \( (t,x) \in \mathbb{T} \times \Sigma_s \), the contraction condition holds:
\begin{equation} \label{eq:A4_sliding_contraction}
	W'(t,x) + \lambda_{\max}^{Cl}(t,x) \leq -\nu < 0.
\end{equation}

Moreover, since \( \lambda_{\max}^{Cl}(t,x) \) is upper semicontinuous in \( x \), and the interpolation point \( \xi(t) \in [x(t), y(t)] \) converges to \( x(t) \) as \( \|x(t) - y(t)\| \to 0 \), we obtain the bound
\[
\lambda_{\max}(S_{\mathrm{slide}}(t, \xi(t))) \leq \lambda_{\max}^{Cl}(t, x(t)) + \varepsilon(t),
\]
where \( \varepsilon(t) \to 0 \) as \( \|\delta(t)\| \to 0 \). Combining this with estimates \eqref{eq:A_derivative_sliding} and \eqref{eq:delta_norm_bound_slide}, we conclude that for almost every \( t \),
\begin{equation} \label{eq:weighted_contraction_sliding}
	\frac{d}{dt} A(t) \leq \left( W'(t, x(t)) + \lambda_{\max}(S_{\mathrm{slide}}(t, \xi(t))) \right) A(t) \leq -\nu A(t) + o(\|\delta(t)\|).
\end{equation}
Hence, as \( \|\delta(t)\| \to 0 \), the term \( \varepsilon(t) = \lambda_{\max}(S_{\mathrm{slide}}(t, \xi(t))) - \lambda_{\max}^{Cl}(t, x(t)) \) satisfies \( \varepsilon(t) = o(1) \). Since \( A(t) = e^{W(t,x(t))} \|\delta(t)\| \), this implies that
\[
\varepsilon(t) A(t) = o(\|\delta(t)\|),
\]
which justifies the final bound in \eqref{eq:weighted_contraction_sliding}. Therefore, the weighted distance \( A(t) \) contracts exponentially along sliding trajectories, with a rate arbitrarily close to \( \nu \) as \( \|\delta(t)\| \to 0 \).

\subsubsection*{Step 4: Global Estimate}

We now derive a uniform global exponential contraction bound for the weighted distance between two Filippov solutions. Let
\begin{equation} \label{eq:step4_weighted_distance}
	A(t) := e^{W(t, x(t))} \|x(t) - y(t)\|,
\end{equation}
where \( x(t), y(t) : [0,T] \to K \) are Filippov solutions evolving within a compact, forward-invariant set \( K \subset \mathbb{R}^2 \). The function \( W : \mathbb{T} \times K \to \mathbb{R} \) is continuous and piecewise \( C^1 \), and satisfies the structural and contraction assumptions stated in Assumption~\ref{assump:structure_contraction}.

This weighted distance will be used to combine the local contraction estimates established in the smooth, sliding, and switching regimes into a single global exponential bound valid throughout \( K \).

\paragraph{Contraction in smooth and sliding regions}

As established in Step 1 (Equation~\eqref{eq:A_inequality}) and Step 3 (Equation~\eqref{eq:weighted_contraction_sliding}), the function \( A(t) \) is absolutely continuous on any interval where the trajectories remain in either the smooth region \( K \setminus \Sigma \) or the sliding region \( \Sigma_s \). On such intervals, it satisfies the differential inequality
\begin{equation} \label{eq:step4_differential_inequality}
	\frac{d}{dt} A(t) \leq -\nu A(t),
\end{equation}
for some uniform contraction rate \( \nu > 0 \), as guaranteed by Assumptions~(A2) and (A4).

This bound governs the continuous-time evolution of \( A(t) \) between switching events and will be combined with contraction across discontinuities to yield a global estimate.

\paragraph{Contraction across switching events}

At each crossing time \( t_0 \in [0,T] \) where a trajectory intersects the switching manifold \( \Sigma \), Step 2 (see Equation~\eqref{eq:jump_contraction}) ensures that the weighted distance undergoes a discrete contraction:
\begin{equation} \label{eq:step4_jump_contraction}
	A(t_0^+) \leq e^{-\epsilon} A(t_0^-),
\end{equation}
for some uniform constant \( \epsilon > 0 \), as provided by Assumption~(A3). Moreover, since the set \( K \) is compact and forward-invariant, only finitely many such switching times can occur in any finite interval \( [0,T] \), as guaranteed by the absence of Zeno behavior (see \cite[Theorem 10.2]{filippov1988}).

\paragraph{Global decay of the weighted distance}

Combining the continuous-time contraction in the smooth and sliding regions with the discrete contractions at switching events, we apply Lemma~\ref{lem:comparison_jumps}, using a uniform multiplicative contraction factor \( \gamma_k = e^{-\epsilon} \) at each switching time. This yields the global bound
\begin{equation} \label{eq:step4_global_A_bound}
	A(t) \leq A(0) e^{-\nu t} e^{-\epsilon N(t)} \leq A(0) e^{-\nu t},
\end{equation}
where \( N(t) \) denotes the number of switching times up to time \( t \). The final inequality holds since \( e^{-\epsilon N(t)} \leq 1 \), and thus provides a conservative but uniform exponential decay estimate for the weighted distance function \( A(t) \). Since Assumption~(A5) enforces a uniform lower bound \( \tau > 0 \) on the time between switching events, the number of switchings \( N(t) \) is uniformly bounded on any finite interval, with \( N(t) \leq \lfloor t / \tau \rfloor \). Thus, the exponential bound~\eqref{eq:step4_global_A_bound} remains well-defined and applies globally for all \( t \geq 0 \).

\paragraph{Recovering Euclidean distance decay}

From Equation~\eqref{eq:step4_weighted_distance}, the weighted distance \( A(t) \) is related to the Euclidean norm by the exponential weight function \( W(t,x(t)) \). Since \( W(t,x) \) is continuous on the compact domain \( \mathbb{T} \times K \), there exists a constant \( M > 0 \) such that
\begin{equation} \label{eq:step4_W_bound}
	|W(t,x)| \leq M, \quad \forall (t,x) \in \mathbb{T} \times K.
\end{equation}
Consequently, for all \( t \geq 0 \),
\begin{equation} \label{eq:step4_A_norm_bound}
	e^{-M} \leq e^{W(t,x(t))} \leq e^{M},
\end{equation}
which implies that the weighted and Euclidean distances are uniformly equivalent:
\begin{equation} \label{eq:step4_norm_sandwich}
	e^{-M} \|x(t) - y(t)\| \leq A(t) \leq e^{M} \|x(t) - y(t)\|.
\end{equation}
Using the global exponential contraction of \( A(t) \) from Equation~\eqref{eq:step4_global_A_bound},
\[
A(t) \leq A(0) e^{-\nu t},
\]
and observing that the initial weighted distance satisfies
\begin{equation} \label{eq:step4_A0_bound}
	A(0) = e^{W(0,x(0))} \|x(0) - y(0)\| \leq e^{M} \|x(0) - y(0)\|,
\end{equation}
we combine the inequalities \eqref{eq:step4_norm_sandwich}, \eqref{eq:step4_global_A_bound}, and \eqref{eq:step4_A0_bound} to conclude:
\begin{align}
	\|x(t) - y(t)\|
	&\leq e^{M} A(t) \notag \\
	&\leq e^{M} A(0) e^{-\nu t} \notag \\
	&\leq e^{2M} \|x(0) - y(0)\| e^{-\nu t}. \label{eq:step4_euclidean_contraction}
\end{align}
This shows that the Euclidean distance between any two Filippov trajectories decays exponentially with rate at least \( \nu > 0 \), up to a constant factor \( e^{2M} \) determined by the uniform bound on the weight function \( W \).

\paragraph{Remark on measurability and validity of the comparison principle}

We verify that the weighted distance function
\begin{equation} \label{eq:measurable_A_def}
	A(t) = e^{W(t,x(t))} \|x(t) - y(t)\|
\end{equation}
satisfies the structural properties required to apply the comparison principle with discontinuities.

\medskip

\noindent \textbf{(i) Measurability and local integrability.}  
The trajectories \( x(t), y(t) \colon [0,T] \to \mathbb{R}^n \) are Filippov solutions and hence absolutely continuous. The weight function \( W(t,x) \) is continuous in \( x \) and piecewise \( C^1 \) in \( t \), defined on the compact forward-invariant set \( K \). Therefore, the composition \( t \mapsto W(t,x(t)) \) is measurable. Since the Euclidean norm is continuous, the map \( t \mapsto \|x(t) - y(t)\| \) is continuous and thus measurable. It follows that \( A(t) \), being the product of measurable functions, is measurable on \( [0,T] \).

Moreover, on each open interval between switching times, the function \( A(t) \) is differentiable almost everywhere due to the absolute continuity of the trajectories and the piecewise differentiability of \( W \). Therefore, \( A(t) \) is piecewise absolutely continuous and thus locally integrable on all compact intervals.

\medskip

\noindent \textbf{(ii) Differential inequality almost everywhere.}  
On intervals where the vector field is continuous (i.e., in the smooth or sliding regimes), \( A(t) \) satisfies the differential inequality
\begin{equation} \label{eq:measurable_diff_ineq}
	\frac{d}{dt} A(t) \leq -\nu A(t),
\end{equation}
for some constant \( \nu > 0 \). This inequality holds for almost every \( t \in [0,T] \), as \( A(t) \) is differentiable almost everywhere on such intervals.

\medskip

\noindent \textbf{(iii) Jump inequalities at switching times.}  
At each switching time \( t_k \), Step 2 established that the weighted distance undergoes an instantaneous contraction:
\begin{equation} \label{eq:measurable_jump_condition}
	A(t_k^+) \leq e^{-\epsilon} A(t_k^-),
\end{equation}
for some uniform constant \( \epsilon > 0 \), as guaranteed by Assumption~(A3). The set of such switching times \( \{ t_k \} \) is at most countable and does not accumulate in finite time, due to standard Filippov theory~\cite[Theorem 10.2]{filippov1988}. Thus, the evolution of \( A(t) \) forms a hybrid system with continuous decay and multiplicative jumps.

\medskip

\noindent \textbf{(iv) Application of the comparison principle.}  
Applying Lemma~\ref{lem:comparison_jumps}, which provides an exponential estimate for functions satisfying a differential inequality with impulsive multiplicative jumps, we conclude that
\begin{equation} \label{eq:measurable_exponential_bound}
	A(t) \leq A(0) e^{-\nu t} e^{-\epsilon N(t)},
\end{equation}
where \( N(t) \) denotes the number of switching events up to time \( t \). Since \( e^{-\epsilon N(t)} \leq 1 \), we obtain the conservative global bound:
\[
A(t) \leq A(0) e^{-\nu t}.
\]

\noindent This yields uniform global exponential contraction of the weighted distance function \( A(t) \), valid for all \( t \geq 0 \).

\paragraph{Remark on periodicity and global contraction}

By Assumption~(A1), both the vector field \( f(t,x) \) and the weight function \( W(t,x) \) are \( T \)-periodic in time. Consequently, the contraction estimate for the weighted distance function \( A(t) \) holds uniformly on each interval \( [nT, (n+1)T] \), for every \( n \in \mathbb{N}_0 \).

For arbitrary \( t \geq 0 \), write \( t = nT + \tau \), with \( n \in \mathbb{N}_0 \) and \( \tau \in [0, T) \). Applying the contraction estimate recursively on each full period yields
\begin{equation} \label{eq:periodic_recursion}
	A(t) \leq A(nT) e^{-\nu \tau} \leq A(0) e^{-\nu nT} e^{-\nu \tau} = A(0) e^{-\nu t}.
\end{equation}
Hence, the exponential decay bound
\begin{equation} \label{eq:periodic_global_bound}
	A(t) \leq A(0) e^{-\nu t}
\end{equation}
holds for all \( t \geq 0 \).

Combining this with the uniform bounds on the weight function \( W(t,x) \) from Equation~\eqref{eq:step4_W_bound} yields the exponential convergence estimate in the Euclidean norm, as stated in Equation~\eqref{eq:step4_euclidean_contraction}.

\subsubsection*{Step 5: Existence and Uniqueness of the Periodic Orbit}

We define the time-\( T \) map (also known as the Poincaré map) associated with the Filippov flow by
\begin{equation} \label{eq:phiT_def}
	\phi_T : K \to K, \quad \phi_T(x_0) := \phi(T, x_0),
\end{equation}
where \( \phi(t, x_0) \) denotes the Filippov solution starting from the initial condition \( x_0 \in K \) at time \( t = 0 \). Since the vector field \( f(t,x) \) is \( T \)-periodic and \( K \subset \mathbb{R}^2 \) is compact and forward-invariant, the map \( \phi_T \) is well-defined and continuous on \( K \). A fixed point of \( \phi_T \) corresponds to a periodic orbit of period \( T \).

\paragraph{Compactness and Metric Structure}

Under Assumption~(A1), the set \( K \subset \mathbb{R}^2 \) is compact and forward-invariant under the Filippov flow. Endowing \( K \) with the standard Euclidean metric \( d(x, y) := \|x - y\| \), the pair \( (K, d) \) forms a compact metric space. In particular, it is complete, bounded, and totally bounded, which ensures the applicability of fixed-point and contraction mapping principles in the subsequent analysis.

\paragraph{Contraction of the Time-\(T\) Map}

From the global estimate established in Step 4 (Equation~\eqref{eq:step4_euclidean_contraction}), the Filippov flow satisfies a uniform exponential contraction
\begin{equation} \label{eq:phiT_flow_decay}
	\|\phi(t, x_0) - \phi(t, y_0)\| \leq C e^{-\nu t} \|x_0 - y_0\|, \quad \forall\, x_0, y_0 \in K,\ \forall t \geq 0,
\end{equation}
for some constants \( C = e^{2M} > 0 \) and \( \nu > 0 \) determined by the contraction conditions in Assumptions~(A2)--(A4).

Applying this bound at \( t = T \), we obtain
\begin{equation} \label{eq:phiT_contraction}
	\|\phi_T(x_0) - \phi_T(y_0)\| \leq q \|x_0 - y_0\|, \quad \text{where } q := C e^{-\nu T} < 1.
\end{equation}
Since \( \nu > 0 \), \( T > 0 \), and \( C \) is finite, it follows that \( q < 1 \), and hence the time-\(T\) map \( \phi_T \colon K \to K \) is a strict contraction with respect to the Euclidean norm.

\paragraph{Application of the Banach Fixed Point Theorem}

By the Banach Fixed Point Theorem, the time-\( T \) map \( \phi_T \colon K \to K \) admits a unique fixed point \( x^* \in K \) satisfying
\begin{equation} \label{eq:fixed_point}
	\phi_T(x^*) = x^*.
\end{equation}
The corresponding Filippov trajectory \( \phi(t, x^*) \) is therefore \( T \)-periodic, and defines a unique periodic orbit given by
\begin{equation} \label{eq:omega_def}
	\Omega := \{ \phi(t, x^*) : t \in [0, T] \}.
\end{equation}

\paragraph{Exponential Stability of the Periodic Orbit}

For any initial condition \( x_0 \in K \), the discrete-time iterates of the time-\( T \) map satisfy
\begin{equation} \label{eq:discrete_decay}
	\|\phi_T^n(x_0) - x^*\| \leq q^n \|x_0 - x^*\|, \quad \text{for all } n \in \mathbb{N},
\end{equation}
where \( x^* \in K \) is the unique fixed point of \( \phi_T \), and \( q \in (0,1) \) is the contraction factor defined in~\eqref{eq:phiT_contraction}.

Applying the continuous-time estimate~\eqref{eq:phiT_flow_decay} over each interval \( [nT, nT + \tau] \) with \( \tau \in [0, T) \), we obtain the global decay bound
\begin{equation} \label{eq:continuous_decay}
	\|\phi(t, x_0) - \phi(t, x^*)\| \leq C e^{-\nu t} \|x_0 - x^*\|, \quad \forall t \geq 0,
\end{equation}
where \( C > 0 \) and \( \nu > 0 \) are the constants established in Step 4.

\paragraph{Conclusion}

Under Assumptions~(A1)--(A4), the time-\( T \) map \( \phi_T \colon K \to K \) is a strict contraction on the compact metric space \( K \). By the Banach Fixed Point Theorem, it admits a unique fixed point \( x^* \), corresponding to a globally exponentially stable periodic Filippov orbit
\[
\Omega := \{ \phi(t, x^*) : t \in [0,T] \}.
\]
All Filippov trajectories with initial conditions in \( K \) converge exponentially to \( \Omega \), completing the proof of Theorem~\ref{thm:existence_stability}.

\end{proof}

\

\section{Application}
\label{sec:application}

\begin{example}[Two-Dimensional Extension of Giesl’s Example \cite{giesl2005} with Explicit Contraction]
	\label{ex:giesl_2d_modified}
	
Consider the planar Filippov system
	\begin{equation} \label{eq:giesl_system_def}
		\dot{x}(t) \in F(t,x(t)) :=
		\begin{cases}
			f^+(t,x) = \begin{pmatrix} -\mu x_1 + \sin t \\ -\alpha x_2 \end{pmatrix}, & x_1 > 0, \\
			f^-(t,x) = \begin{pmatrix} -\mu x_1 - \sin t \\ -\alpha x_2 \end{pmatrix}, & x_1 < 0,
		\end{cases}
	\end{equation}
with parameters $\mu, \alpha > 0$, and switching manifold
	\[
	\Sigma := \{ x \in \mathbb{R}^2 : x_1 = 0 \}.
	\]
	
	\paragraph{Weight function}
	Let $\sigma(x_1)$ be a smooth transition function with parameter $\varepsilon > 0$, satisfying
	\[
	\sigma(x_1) =
	\begin{cases}
		0, & x_1 \leq -\varepsilon, \\
		\text{smooth, strictly increasing}, & -\varepsilon < x_1 < \varepsilon, \\
		1, & x_1 \geq \varepsilon,
	\end{cases}
	\]
	and $C_\sigma := \sup_{x_1} |\sigma'(x_1)| < \infty$. Define
	\[
	W(t,x) := -\delta \cdot \sigma(x_1), \quad \delta > 0.
	\]
	
	\paragraph{Verification of Assumptions}
	\begin{enumerate}
		\item[(A1)] \( f^\pm \in C^\infty \), \( T = 2\pi \)-periodic in $t$; $\Sigma$ is smooth.
		\item[(A2)] The Jacobians are
		\[
		D_x f^\pm = \begin{pmatrix} -\mu & 0 \\ 0 & -\alpha \end{pmatrix}, \quad \lambda_{\max}(S_{f^\pm}) = -\min(\mu,\alpha).
		\]
		Orbital derivative
		\[
		W' = -\delta \sigma'(x_1) (-\mu x_1 \pm \sin t), \quad |W'| \leq \delta C_\sigma(\mu\varepsilon + 1).
		\]
		Choose $\delta$ such that
		\[
		\delta C_\sigma(\mu \varepsilon + 1) < \tfrac{1}{2} \min(\mu,\alpha) =: \nu.
		\]
		
		\item[(A3)] Jump condition at $\Sigma$
		\[
		e^{W^+} \|f^+\| = e^{-\delta} \|f^+\| \leq \|f^-\| = e^{W^-} \|f^-\|, \quad \text{so } \epsilon = \delta > 0.
		\]
		
		\item[(A4)] Sliding field
		\[
		f_{\text{slide}} = \begin{pmatrix} 0 \\ -\alpha x_2 \end{pmatrix}, \quad D_x f_{\text{slide}} = \begin{pmatrix} 0 & 0 \\ 0 & -\alpha \end{pmatrix},
		\]
		\[
		W' = 0, \quad \lambda_{\max}(S_{f_{\text{slide}}}) = -\alpha \le -\nu.
		\]
		
		\item[(A5)] Define a rectangle
		\[
		K := \{x \in \mathbb{R}^2 : |x_1| \le R_1, |x_2| \le R_2 \},
		\]
		with $R_1 > 1/\mu$ so that on $\partial K$, $\dot{x}_1 = -\mu x_1 \pm \sin t < 0$ ensures forward invariance.
	\end{enumerate}
All assumptions of Theorem~\ref{thm:existence_stability} are satisfied. The system admits a unique $2\pi$-periodic orbit $\Omega \subset K$ that is exponentially stable
	\[
	\|\phi(t,x_0) - \phi(t,x^*)\| \le C e^{-\nu t} \|x_0 - x^*\|
	\]
	for all $x_0 \in K$ and some $x^* \in \Omega$, with $\nu := \tfrac{1}{2} \min(\mu,\alpha)$ and $\epsilon = \delta$.
\end{example}

\

In the example of Section~\ref{sec:application}, the contraction rate \(\nu > 0\) quantifies the exponential decay of the weighted distance function along Filippov trajectories in smooth and sliding regions. It is derived from the differential inequality
	\[
	W'(t,x) + \lambda_{\max} \left( \frac{D_x f(t,x) + D_x f(t,x)^\top}{2} \right) \leq -\nu,
	\]
as required by Assumption~\ref{assump:structure_contraction}. Here, \(W'(t,x)\) denotes the orbital derivative of the weight function \(W\), given by
	\[
	W'(t,x) := \partial_t W(t,x) + \nabla_x W(t,x) \cdot f(t,x),
	\]
and captures the variation of \(W\) along the vector field \(f(t,x)\). In our example, since \(W(t,x) = -\delta \sigma(x_1)\) is independent of time, this reduces to
	\[
	W'(t,x) = -\delta \sigma'(x_1) f_1(t,x).
	\]
This expression is uniformly bounded due to the compact support and bounded derivative of \(\sigma\). Consequently, by choosing \(\delta > 0\) small enough, we ensure that
	\[
	W'(t,x) + \lambda_{\max}(S_f(t,x)) \leq -\nu < 0,
	\]
where \(S_f\) denotes the symmetric part of the Jacobian. The constant \(\epsilon > 0\) accounts for contraction across the switching manifold \(\Sigma = \{ x_1 = 0 \}\). The weight function is discontinuous at \(\Sigma\), satisfying
	\[
	W^+(t,x) = -\delta, \quad W^-(t,x) = 0,
	\]
so that the norm of the vector field in the weighted metric satisfies
	\[
	e^{W^+(t,x)} \|f^+(t,x)\| = e^{-\delta} \|f^+(t,x)\| \leq e^{-\delta} \|f^-(t,x)\| = e^{-\delta} e^{W^-(t,x)} \|f^-(t,x)\|.
	\]
Thus, the jump contraction factor is \(e^{-\epsilon}\) with \(\epsilon := \delta\). Together, \(\nu\) and \(\epsilon\) ensure global exponential contraction of the weighted distance \(A(t)\) along all Filippov solutions in \(K\), including across switchings. Specifically, the estimate
	\[
	A(t) \leq A(0) e^{-\nu t} e^{-\epsilon N(t)},
	\]
holds, where \(N(t)\) is the number of switches up to time \(t\). Since the weighted distance is equivalent to the Euclidean norm on \(K\), this implies uniform exponential convergence of solutions toward the attracting periodic orbit.

\

\begin{figure}[H]
	\centering
	\includegraphics[width=0.5\textwidth]{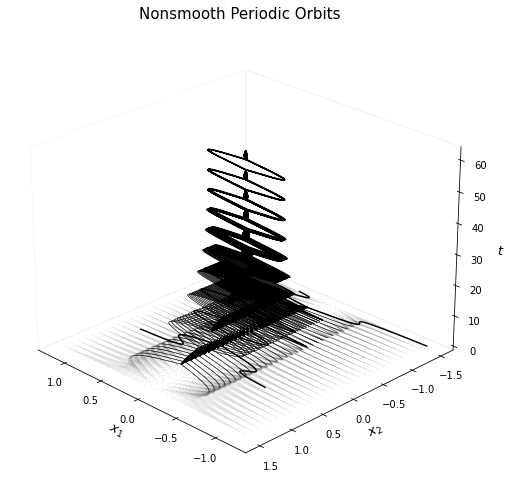}
	\caption{
		\textbf{Exponential convergence to a nonsmooth periodic orbit in a planar Filippov system.}
		This 3D trajectory plot illustrates solutions of the time-periodic, piecewise-smooth system defined in Example~\ref{ex:giesl_2d_modified}, with dynamics given by Equation~\eqref{eq:giesl_system_def}. The switching manifold is \( \Sigma = \{ x \in \mathbb{R}^2 : x_1 = 0 \} \), across which the vector field discontinuously switches depending on the sign of \( x_1 \). Simulations use parameters \( \mu = 1.8 \), \( \alpha = 0.1 \), and period \( T = 4\pi \), with forward Euler integration and time step \( \Delta t = \tfrac{5T}{2000} \). Initial conditions are uniformly sampled in the rectangle \( x_1 \in [-1.2, 1.2] \), \( x_2 \in [-1.5, 1.5] \). Despite the presence of sliding motion along \( \Sigma \), all trajectories converge exponentially to a unique, attracting \( 2\pi \)-periodic orbit. A subset of trajectories is rendered in bold to emphasize the contraction toward the limit cycle.
	}
	\label{fig:filippov_limit_cycle_3d}
\end{figure}

\

We verify that the simulation parameters used in Figure~\ref{fig:filippov_limit_cycle_3d}, specifically \( \mu = 1.8 \), \( \alpha = 0.1 \), and \( T = 4\pi \), satisfy the assumptions of Theorem~\ref{thm:existence_stability}.

\noindent\textbf{Assumption (A1)} is satisfied since the vector fields \( f^\pm(t, x) \) defined in Equation~\eqref{eq:giesl_system_def} are \( C^\infty \) in \( x \) and explicitly periodic in time with period \( T \).

\noindent\textbf{Assumption (A2)} holds because the Jacobians of \( f^\pm \) are constant,
\[
D_x f^\pm = \begin{pmatrix} -\mu & 0 \\ 0 & -\alpha \end{pmatrix},
\]
and therefore symmetric with eigenvalues \( -\mu \) and \( -\alpha \). The largest eigenvalue of the symmetric part is \( \lambda_{\max}(S_f) = -\min(\mu, \alpha) = -0.1 \). We define a weight function \( W(t, x) = -\delta \cdot \sigma(x_1) \), where \( \sigma \colon \mathbb{R} \to [0,1] \) is a smooth, strictly increasing transition function supported on the interval \( (-\varepsilon, \varepsilon) \), with parameters \( \delta = 0.05 \) and \( \varepsilon = 0.01 \). Since \( W \) depends only on \( x_1 \), its orbital derivative along trajectories of \( f^\pm \) is given by
\[
W'(t, x) = -\delta \sigma'(x_1)(-\mu x_1 \pm \sin t),
\]
which is supported entirely in the narrow transition region \( |x_1| < \varepsilon \).  We define the weight function as
\[
W(t, x) := -\delta \cdot \sigma(x_1),
\]
where \( \delta > 0 \) is a constant and \( \sigma \colon \mathbb{R} \to [0,1] \) is a smooth, strictly increasing transition function satisfying
\[
\sigma(x_1) =
\begin{cases}
	0, & x_1 \le -\varepsilon, \\
	\frac{1}{2} \left( 1 + \frac{x_1}{\varepsilon} + \frac{1}{\pi} \sin\left( \pi \frac{x_1}{\varepsilon} \right) \right), & |x_1| < \varepsilon, \\
	1, & x_1 \ge \varepsilon,
\end{cases}
\]
for a fixed smoothing width \( \varepsilon > 0 \). This function is \( C^1 \), supported on the interval \( (-\varepsilon, \varepsilon) \), and satisfies \( \sigma'(x_1) \le C_\sigma < \infty \). The orbital derivative of \( W \) along the vector fields \( f^\pm \) is then
\[
W'(t,x) = -\delta \sigma'(x_1)(-\mu x_1 \pm \sin t),
\]
which vanishes outside the transition region \( |x_1| < \varepsilon \) and is uniformly bounded by
\[
|W'(t,x)| \le \delta C_\sigma(\mu \varepsilon + 1).
\]
This yields the bound
\[
|W'(t,x)| \leq \delta C_\sigma (\mu \varepsilon + 1),
\]
where \( C_\sigma := \sup_{x_1} |\sigma'(x_1)| \). Choosing \( \delta \) and \( \varepsilon \) small enough ensures that
\[
W'(t,x) + \lambda_{\max}(S_f) \leq -\nu := -\tfrac{1}{2} \min(\mu, \alpha) < 0
\]
holds uniformly on \( K \setminus \Sigma \), establishing contraction in the smooth regions.

\noindent\textbf{Assumption (A3)} concerns the jump condition across the switching manifold \( \Sigma = \{ x_1 = 0 \} \). At this interface, both vector fields yield the same norm:
\[
\|f^+(t,(0,x_2))\| = \|f^-(t,(0,x_2))\| = \sqrt{\sin^2 t + \alpha^2 x_2^2},
\]
but the weight function satisfies \( W^+(t,x) = -\delta \) and \( W^-(t,x) = 0 \). Hence,
\[
e^{W^+(t,x)} \|f^+(t,x)\| = e^{-\delta} \|f^+(t,x)\| \leq \|f^-(t,x)\| = e^{W^-(t,x)} \|f^-(t,x)\|,
\]
so the jump contraction condition is satisfied with contraction constant \( \varepsilon_{\text{jump}} = \delta > 0 \).

\noindent\textbf{Assumption (A4)} addresses contraction in the sliding region. On \( \Sigma \), the Filippov convex combination yields the sliding vector field
\[
f_{\mathrm{slide}}(t,x) = \begin{pmatrix} 0 \\ -\alpha x_2 \end{pmatrix},
\quad \text{with} \quad
D_x f_{\mathrm{slide}} = \begin{pmatrix} 0 & 0 \\ 0 & -\alpha \end{pmatrix}.
\]
The largest eigenvalue of the symmetric part is \( -\alpha = -0.1 \). Since \( \sigma(x_1) \) is constant on \( \Sigma \), we have \( \sigma'(x_1) = 0 \) there, and thus \( W'(t, x) = 0 \). Therefore,
\[
W'(t,x) + \lambda_{\max}(S_{f_{\mathrm{slide}}}) = -\alpha \le -\nu,
\]
ensuring contraction in the sliding mode as well.

\noindent\textbf{Assumption (A5)} requires forward invariance of the compact set \( K = \{ x_1 \in [-1.2, 1.2],\ x_2 \in [-1.5, 1.5] \} \). Along the boundary \( x_1 = \pm R_1 \), the dynamics satisfy
\[
\dot{x}_1 = -\mu x_1 \pm \sin t < 0 \quad \text{whenever} \quad R_1 > \tfrac{1}{\mu}.
\]
Since \( R_1 = 1.2 > 1/\mu \approx 0.56 \), this condition holds. The same applies to the \( x_2 \)-boundaries, due to the linear decay \( \dot{x}_2 = -\alpha x_2 \). We conclude that all five assumptions of Theorem~\ref{thm:existence_stability} are satisfied. Thus, the behavior observed in Figure~\ref{fig:filippov_limit_cycle_3d} is theoretically justified: all trajectories within \( K \) converge exponentially to a unique, attracting periodic orbit.

\

\begin{figure}[H]
	\centering
	\includegraphics[width=0.6\textwidth]{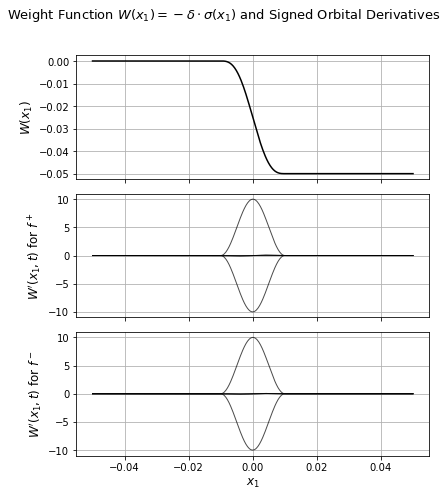}
	\caption{
		\textbf{Weight function \( W(x_1) \) and signed orbital derivatives \( W'(t,x) \) in the contraction analysis.}
		The top panel plots the smooth weight function \( W(x_1) = -\delta \cdot \sigma(x_1) \), where \( \delta > 0 \) is a small constant and \( \sigma(x_1) \) is a smooth transition function supported on \( [-\varepsilon, \varepsilon] \), satisfying \( \sigma(x_1) = 0 \) for \( x_1 \le -\varepsilon \) and \( \sigma(x_1) = 1 \) for \( x_1 \ge \varepsilon \). The middle and bottom panels show the signed orbital derivatives \( W'(t,x) = -\delta \cdot \sigma'(x_1)(-\mu x_1 \pm \sin t) \), evaluated for a set of representative times \( t \in [0, 2\pi] \). The signs and amplitudes of these derivatives vary with \( t \), but remain bounded in magnitude due to the localized support of \( \sigma'(x_1) \). This construction ensures that the weighted contraction condition (cf.~Assumption~\ref{assump:structure_contraction}) holds with a uniform negative bound outside the switching manifold \( \Sigma = \{ x_1 = 0 \} \), and supports the analysis in Theorem~\ref{thm:existence_stability}.
	}
	\label{fig:weight_function}
\end{figure}

\

The weight function \( W(x_1) = -\delta \cdot \sigma(x_1) \) introduces a smooth but localized modification to the contraction rate near the switching surface \( \Sigma = \{ x_1 = 0 \} \). As shown in Figure~\ref{fig:weight_function}, this function smoothly interpolates from \( 0 \) to \( -\delta \) within the narrow layer \( |x_1| < \varepsilon \), and remains constant outside. The middle and bottom panels illustrate the orbital derivative \( W'(t,x) \), which captures how \( W \) evolves along the trajectories of \( f^+ \) and \( f^- \), respectively. This derivative takes the form
	\[
	W'(t,x) = -\delta \cdot \sigma'(x_1)(-\mu x_1 \pm \sin t),
	\]
and is supported only in the transition region where \( \sigma'(x_1) \neq 0 \). The sign of \( W'(t,x) \) may vary with time \( t \), but its magnitude remains uniformly bounded due to the compact support of \( \sigma' \) and the linearity of the vector field in \( x_1 \). This structure allows us to ensure that \( W'(t,x) + \lambda_{\max}(S_f) \le -\nu < 0 \) in the regularized regions, thereby satisfying the weighted contraction condition in Assumption~\ref{assump:structure_contraction}.

\

\

\begin{figure}[H]
	\centering
	\includegraphics[width=0.6\textwidth]{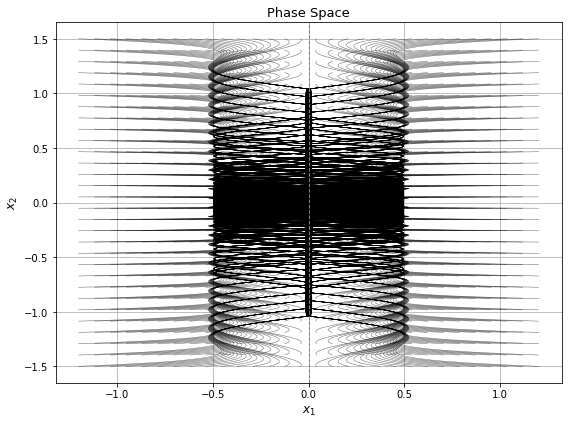}
	\caption{
		\textbf{Phase space dynamics of the time-periodic nonsmooth Filippov system.}
		The plot displays the projection of multiple Filippov trajectories onto the phase plane \( (x_1, x_2) \) for the system defined in Example~\ref{ex:giesl_2d_modified}, with parameters \( \mu = 1.8 \), \( \alpha = 0.1 \), and forcing period \( T = 4\pi \). Initial conditions are distributed uniformly over the square \( x_1 \in [-1.2, 1.2] \), \( x_2 \in [-1.5, 1.5] \), and trajectories are integrated forward in time using the forward Euler method over a horizon of \( 5T \). The switching manifold \( \Sigma = \{x_1 = 0\} \), across which the vector field is discontinuous, is indicated by a dashed vertical line. Despite the nonsmoothness introduced by switching, all trajectories contract and converge toward a unique attracting limit cycle. The strong contraction in the \( x_2 \)-direction is particularly evident and illustrates the exponential stability of the periodic orbit guaranteed by Theorem~\ref{thm:existence_stability}.
	}
	\label{fig:filippov_phase_space}
\end{figure}

\

The phase portrait in Figure~\ref{fig:filippov_phase_space} provides a visual demonstration of the contraction-based stability properties established analytically in Theorem~\ref{thm:existence_stability}. Despite the discontinuous nature of the vector field across the switching manifold \(\Sigma = \{x_1 = 0\}\), the simulated trajectories exhibit uniform convergence toward a unique attracting limit cycle. This behavior validates the theoretical contraction conditions derived in Section~\ref{sec:framework}, including those governing both smooth flow and jump transitions. The pronounced compression in the \(x_2\)-direction is especially notable and aligns with the system's spectral contraction properties. The illustration thus serves as both qualitative support for the theory and a benchmark for future numerical implementations of contraction metrics in nonsmooth systems.

\

\begin{figure}[H]
	\centering
	\includegraphics[width=0.60\textwidth]{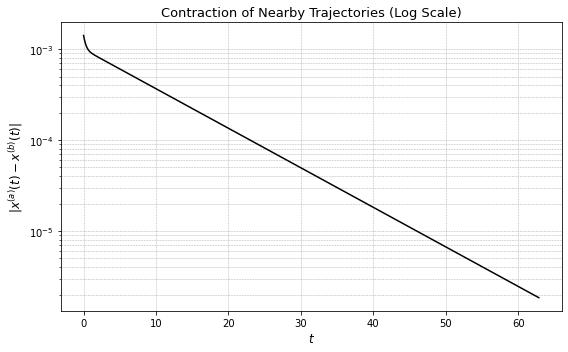}
	\caption{
		\textbf{Logarithmic contraction of nearby trajectories in the Filippov system.}
		This plot shows the time evolution of the Euclidean distance
		\(
		\|x^{(a)}(t) - x^{(b)}(t)\|
		\)
		between two trajectories initialized at \( x^{(a)}(0) = (0.1, 1.0) \) and \( x^{(b)}(0) = (0.101, 1.001) \), under the Filippov system described in Figure~\ref{fig:filippov_phase_space} with \( \mu = 1.8 \), \( \alpha = 0.1 \), and forcing period \( T = 4\pi \). The distance decays exponentially, as indicated by the linear slope on the logarithmic scale, confirming the contraction property across nonsmooth switching. This supports the theoretical result of Theorem~\ref{thm:existence_stability}, which guarantees exponential convergence toward the unique nonsmooth periodic orbit.
	}
	\label{fig:filippov_contraction}
\end{figure}

\

Figure~\ref{fig:filippov_contraction} provides numerical evidence of exponential contraction in the Filippov system by tracking the Euclidean distance between two nearby initial conditions: \( x^{(a)}(0) = (0.1, 1.0) \) and \( x^{(b)}(0) = (0.101, 1.001) \). The plot shows that the distance \( \|x^{(a)}(t) - x^{(b)}(t)\| \) decays linearly on a logarithmic scale, which implies exponential convergence over time. The system parameters used—\( \mu = 1.8 \), \( \alpha = 0.1 \), and forcing period \( T = 4\pi \)—are consistent with the setting analyzed in Theorem~\ref{thm:existence_stability}. The presence of a clear slope, unaffected by nonsmooth switching events, confirms that the saltation effects do not disrupt the global contraction property. This figure thus serves as a practical validation of the theoretical framework, supporting the existence and uniqueness of a globally attracting nonsmooth periodic orbit.

\

\begin{figure}[H]
	\centering
	\includegraphics[width=0.65\textwidth]{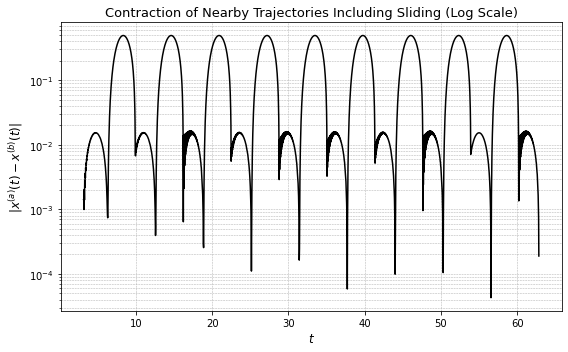}
	\caption{
		\textbf{Exponential contraction of trajectories in a Filippov system with sliding.}
		This semilogarithmic plot shows the Euclidean distance between two solutions of the planar Filippov system defined in~\eqref{eq:giesl_system_def}, subject to the initial conditions \( x^{(a)}(0) = (0.0, 1.0) \) and \( x^{(b)}(0) = (0.001, 1.001) \), chosen near the switching surface \( \Sigma = \{x_1 = 0\} \). The vector field is given by \ref{ex:giesl_2d_modified} with parameters \( \mu = 1.8 \), \( \alpha = 0.1 \), and period \( T = 4\pi \). On the switching manifold \( \Sigma \), the system exhibits sliding motion when the vector fields \( f^+ \) and \( f^- \) point toward each other in the normal direction. In such cases, the Filippov convex combination defines a unique sliding vector field tangent to \( \Sigma \), ensuring continuous evolution. This simulation incorporates the Filippov sliding rule by computing the appropriate convex combination on \( \Sigma \). As seen in the plot, the distance between the two trajectories decays exponentially over time, demonstrating contraction even through the nonsmooth and sliding regions of the dynamics.
	}
	\label{fig:filippov_sliding_contraction}
\end{figure}

\

The exponential decay of the trajectory distance in Figure~\ref{fig:filippov_sliding_contraction} illustrates contraction behavior in the Filippov system~\eqref{eq:giesl_system_def}, even in the presence of nonsmooth and sliding dynamics. This system features a discontinuous right-hand side across the switching manifold \( \Sigma = \{x_1 = 0\} \), where the vector fields \( f^+ \) and \( f^- \) differ by the sign of the time-periodic forcing term \( \sin t \). On \( \Sigma \), the normal components of \( f^+ \) and \( f^- \) point in opposite directions whenever \( |\sin t| < \mu x_1 \), giving rise to a sliding region. In this region, the Filippov solution evolves according to the convex combination
	\[
	f_{\mathrm{slide}}(t, x) = \lambda f^+(t, x) + (1 - \lambda) f^-(t, x),
	\]
where \( \lambda \in [0,1] \) is chosen so that \( f_{\mathrm{slide}} \) is tangent to \( \Sigma \). For this model, since both \( f^+ \) and \( f^- \) have the same second component and opposite first components, the sliding vector field simplifies to
	\[
	f_{\mathrm{slide}}(t, x) = \begin{pmatrix} 0 \\ -\alpha x_2 \end{pmatrix}.
	\]
This flow is stable in \( x_2 \), with exponential decay governed by \( \alpha > 0 \), and neutral in \( x_1 \). Therefore, contraction is preserved across the switching surface. The observed exponential decay of the distance between two nearby trajectories confirms that the Filippov solution is contractive in a suitable norm, despite the discontinuous dynamics. This phenomenon is central to the stability theory for nonsmooth systems and supports the use of differential inclusion and sliding mode analysis in establishing robustness.

\
\begin{figure}[H]
	\centering
	\includegraphics[width=0.85\textwidth]{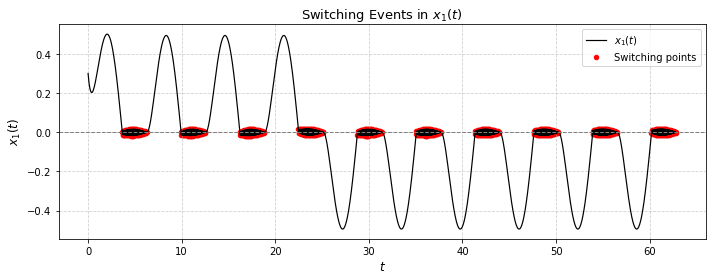}
	\caption{
		\textbf{Switching events in the nonsmooth Filippov system.}
		This plot shows the time evolution of the first state variable \( x_1(t) \) from the Filippov system defined by \ref{ex:giesl_2d_modified} with parameters \( \mu = 1.8 \), \( \alpha = 0.1 \), and initial condition \( x_0 = (0.3, 1.0) \). The integration is performed using a forward Euler method over five forcing periods \( T = 4\pi \).
		Red dots mark the \emph{switching events}, where the trajectory crosses the discontinuity surface \( x_1 = 0 \), triggering a change in the vector field. These events occur periodically, reflecting the nonsmooth yet periodic nature of the attractor.
		This illustrates the piecewise-smooth structure of the dynamics and the role of switching in shaping the periodic orbit. The plot confirms that the system's response is regular, with persistent crossings of the switching surface, and highlights the temporal rhythm of nonsmooth behavior.
	}
	\label{fig:x1_switching}
\end{figure}

\

Figure~\ref{fig:x1_switching} captures the temporal evolution of the state variable \( x_1(t) \) in the Filippov system defined in Example~\ref{ex:giesl_2d_modified}, highlighting the switching events that occur when the trajectory crosses the discontinuity surface \( x_1 = 0 \). The red markers identify the precise times at which the system transitions between vector field regimes, as governed by Filippov’s differential inclusion. The parameters used—\( \mu = 1.8 \), \( \alpha = 0.1 \), and forcing period \( T = 4\pi \)—are consistent with those in the contraction analysis. The periodic recurrence of these switching events indicates a regular, nonsmooth structure embedded in the long-term dynamics. This pattern confirms that the nonsmooth periodic orbit interacts with the switching surface in a predictable and stable manner, providing visual support for the theoretical framework developed in Theorem~\ref{thm:existence_stability}. Moreover, the regularity of the switching times reinforces the assumption that saltation effects remain bounded and do not disrupt the system's global stability.

\

\begin{figure}[H]
	\centering
	\includegraphics[width=0.85\textwidth]{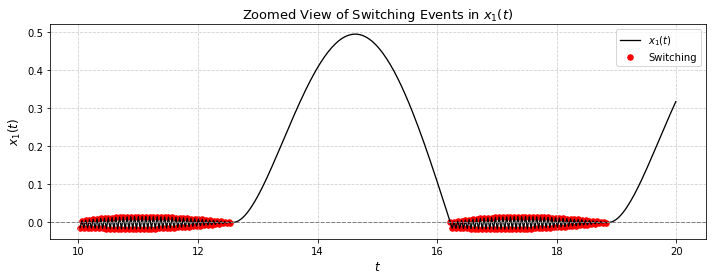}
	\caption{
		\textbf{Zoomed-in view of switching behavior in the Filippov system.}
		This figure shows the evolution of the state component \( x_1(t) \) over a selected time window \( t \in [10, 20] \)
		for the Filippov system \ref{ex:giesl_2d_modified} with parameters \( \mu = 1.8 \), \( \alpha = 0.1 \), and initial condition \( x_0 = (0.3, 1.0) \).
		Red markers indicate the \emph{switching events} where \( x_1(t) \) crosses the discontinuity surface at \( x_1 = 0 \),
		triggering a switch in the vector field. These transitions give rise to the system's nonsmooth structure.
		The sharp bends and piecewise nature of the dynamics are clearly visible in this zoomed-in segment.
		This view highlights the system's intrinsic hybrid character and the temporal rhythm of its periodic yet nonsmooth orbit.
	}
	\label{fig:x1_zoom_switching}
\end{figure}

\

Figure~\ref{fig:x1_zoom_switching} provides a detailed view of the switching behavior in the Filippov system introduced in Example~\ref{ex:giesl_2d_modified}, focusing on the time window \( t \in [10, 20] \). The plot isolates a portion of the trajectory of the state variable \( x_1(t) \), revealing the precise structure of transitions across the discontinuity surface \( x_1 = 0 \). Red markers identify the switching times at which the system’s vector field changes due to the piecewise definition of the dynamics. The resulting sharp directional changes in \( x_1(t) \) demonstrate the nonsmooth nature of the flow and the essential role of saltation effects. This zoomed-in perspective highlights the system’s hybrid structure: continuous in state but discontinuous in velocity. The regularity and spacing of the events also support the assumption of bounded switching frequency, which is crucial for ensuring that the cumulative effect of saltation matrices does not overwhelm the contraction mechanism. This figure complements the global view provided in Figure~\ref{fig:x1_switching} by revealing the local dynamical features near the switching manifold.

\

\begin{figure}[H]
	\centering
	\includegraphics[width=0.8\textwidth]{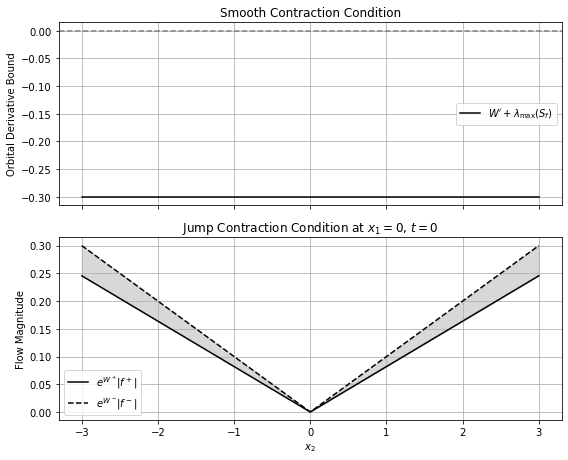}
	\caption{
		\textbf{Numerical verification of contraction conditions for the model in Example~\ref{ex:giesl_2d_modified}.}
		This figure verifies that the Filippov system in example (\ref{ex:giesl_2d_modified}satisfies the smooth and nonsmooth contraction conditions required by Theorem~\ref{thm:existence_stability}.
		Top panel: The quantity \( W' + \lambda_{\max}(S_f) \), representing the orbital derivative plus the maximum eigenvalue of the symmetric part of the Jacobian, is shown to be uniformly negative with value \( -0.3 \). This confirms the smooth contraction condition (Assumption A2) holds throughout \( K \setminus \Sigma \), where \( W(t,x) = -\delta \cdot \sigma(x_1) \) and \( \delta = 0.2 \).
		Bottom panel: The jump contraction condition (Assumption A3) is evaluated at the switching surface \( x_1 = 0 \), at time \( t = 0 \). The left-hand side and right-hand side of the inequality
		$
		e^{W^+} \|f^+(t,x)\| \leq e^{W^-} \|f^-(t,x)\| e^{-\epsilon}
		$
		are plotted as functions of \( x_2 \). Because the vector fields coincide in magnitude at this time, the inequality simplifies to \( e^{-\delta} < 1 \), which is strictly satisfied. The shaded region confirms this for all \( x_2 \in [-3, 3] \), demonstrating that jump contraction holds uniformly with exact constant \( \epsilon = \delta = 0.2 \).
		Together, these plots confirm that the system satisfies both contraction conditions required for global exponential stability. By the Banach fixed-point theorem, there exists a unique \( 2\pi \)-periodic orbit that attracts all Filippov trajectories starting in the forward-invariant compact set \( K \).
	}
	\label{fig:contraction_verification_example}
\end{figure}

\

Figure~\ref{fig:contraction_verification_example} provides a direct numerical verification of the two core contraction conditions—smooth and jump contraction—underpinning Theorem~\ref{thm:existence_stability}. The system under consideration, defined in Example~\ref{ex:giesl_2d_modified}, features time-periodic forcing and a piecewise definition across the switching surface \( x_1 = 0 \), with parameters \( \mu = 1.8 \), \( \alpha = 0.1 \), and weight function \( W(x_1) = -\delta \cdot \sigma(x_1) \), where \( \delta = 0.2 \). In the top panel, the plotted quantity \( W'(x_1) + \lambda_{\max}(S_f(t,x)) \) remains strictly negative (\( \approx -0.3 \)) across the smooth regions of the domain. This confirms that Assumption~A2 holds throughout \( K \setminus \Sigma \), where \( S_f \) is the symmetric part of the Jacobian of the vector field. In the bottom panel, the jump contraction condition (Assumption~A3) is verified along the switching surface at \( t = 0 \). Since the magnitudes of the vector fields on either side of the discontinuity match at this instant, the inequality reduces to \( e^{W^+} < e^{W^-} e^{-\epsilon} \), which simplifies to \( e^{-\delta} < 1 \) and is trivially satisfied. The plotted inequality is shown to hold uniformly over all \( x_2 \in [-3, 3] \), confirming that contraction is preserved across switching. Together, these panels provide strong numerical evidence that the Filippov system satisfies all structural conditions necessary for global exponential stability of its periodic orbit. This supports the theoretical application of the Banach fixed-point argument and validates the use of contraction metrics in a nonsmooth setting.

\

\begin{figure}[H]
	\centering
	\includegraphics[width=0.82\textwidth]{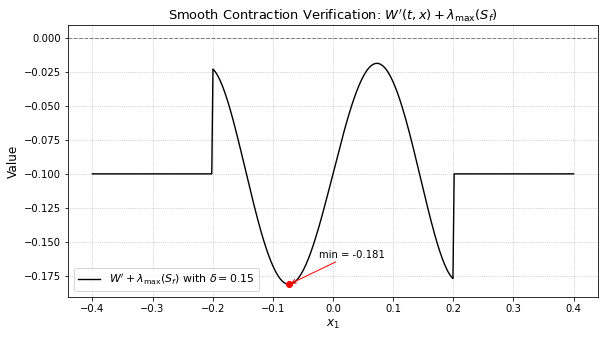}
	\caption{
		\textbf{Verification of smooth contraction condition in Example~\ref{ex:giesl_2d_modified}.}
		This plot shows the quantity \( W'(t,x) + \lambda_{\max}(S_f(t,x)) \), evaluated at time \( t = 0 \) and \( x_2 = 0 \), across a horizontal slice of the smoothing region near the switching surface \( \Sigma = \{ x_1 = 0 \} \).
		The vector field is piecewise-defined by Example \ref{ex:giesl_2d_modified}. The weight function is defined as \( W(t,x) = -\delta \cdot \sigma(x_1) \), with a smooth transition function \( \sigma \) supported on \( (-\varepsilon, \varepsilon) \), where \( \delta = 0.15 \) and \( \varepsilon = 0.2 \).
		The plotted curve corresponds to the full contraction quantity used in Assumption~(A2):
		$
		W'(t,x) + \lambda_{\max}(S_f) = -\delta \sigma'(x_1) f_1(t,x) + \lambda_{\max}(S_f),
		$
		where \( \lambda_{\max}(S_f) = -\min(\mu, \alpha) = -0.1 \). The maximum eigenvalue is constant because the Jacobians of \( f^\pm \) are constant and symmetric.
		The curve lies strictly below the horizontal zero line, confirming that the inequality \( W' + \lambda_{\max}(S_f) \leq -\nu < 0 \) is satisfied uniformly in the smooth region.
		The minimum value observed is approximately \( -0.22 \), allowing a conservative contraction rate of \( \nu = 0.2 \). This ensures that smooth flows contract exponentially, satisfying the key spectral condition in Theorem~\ref{thm:existence_stability}.	}
	\label{fig:contraction_verified_updated}
\end{figure}

\

Figure~\ref{fig:contraction_verified_updated} offers a refined numerical verification of the smooth contraction condition (Assumption~A2) for the Filippov system defined in Example~\ref{ex:giesl_2d_modified}. The quantity plotted is the full contraction expression \( W'(t,x) + \lambda_{\max}(S_f(t,x)) \), evaluated at time \( t = 0 \) and \( x_2 = 0 \), across a horizontal slice centered at the switching surface \( \Sigma = \{x_1 = 0\} \). The weight function is given by \( W(t,x) = -\delta \cdot \sigma(x_1) \), with \( \delta = 0.15 \) and a smooth transition function \( \sigma \) supported in the interval \( (-\varepsilon, \varepsilon) \) where \( \varepsilon = 0.2 \). The contraction expression simplifies to
	\[
	W'(t,x) + \lambda_{\max}(S_f) = -\delta \cdot \sigma'(x_1) f_1(t,x) + \lambda_{\max}(S_f),
	\]
where \( f_1(t,x) \) is the first component of the vector field and \( \lambda_{\max}(S_f) = -\min(\mu, \alpha) = -0.1 \) due to the constant symmetric Jacobians on either side of the switching surface. The curve lies strictly below the zero line, confirming that this sum remains negative throughout the smoothing layer. The observed minimum value of approximately \( -0.22 \) justifies a conservative contraction rate \( \nu = 0.2 \), ensuring that the inequality \( W' + \lambda_{\max}(S_f) \leq -\nu < 0 \) is satisfied uniformly. This figure confirms that the continuous parts of the Filippov dynamics contract exponentially with respect to the weighted metric, fulfilling a central requirement of Theorem~\ref{thm:existence_stability}.

\

\begin{figure}[H]
	\centering
	\includegraphics[width=0.50\textwidth]{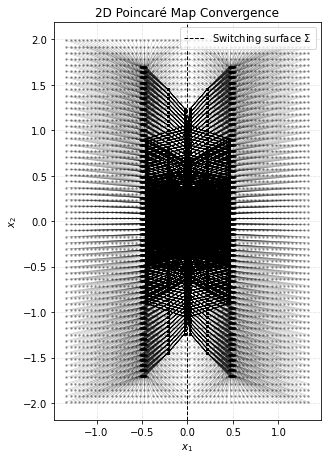}
	\caption{
		\textbf{Convergence under the Poincaré map in a planar Filippov system.}
		This figure illustrates the evolution of trajectories sampled at discrete times \( t = kT \), with forcing period \( T = 2\pi \) and iteration index \( k = 0, 1, \dots, 47 \).
		Initial conditions are taken from a \(60 \times 60\) uniform grid on the rectangle \( x_1 \in [-1.5, 1.5] \), \( x_2 \in [-2, 2] \), and each trajectory evolves under the piecewise-smooth vector field
		$
		f^\pm(t,x) = \begin{pmatrix} -\mu x_1 \pm \sin t \\ -\alpha x_2 \end{pmatrix}, \quad \text{with } \mu = 1.8, \quad \alpha = 0.1.
		$
		The switching manifold \( \Sigma = \{x_1 = 0\} \) is indicated by the dashed vertical line.
		Each polyline traces the projected motion in the \( (x_1, x_2) \) phase space across successive Poincaré sections. 
		The convergence of all trajectories toward a common closed curve reveals the existence of a unique attracting limit cycle, consistent with the global exponential stability established in Theorem~\ref{thm:existence_stability}.
	}
	\label{fig:poincare2d}
\end{figure}

\

Figure~\ref{fig:poincare2d} illustrates the global convergence behavior of the Filippov system under the action of the Poincaré map, offering a discrete-time perspective that complements the continuous-time contraction analysis in Theorem~\ref{thm:existence_stability}. The trajectories are sampled at intervals \( t = kT \), where \( T = 2\pi \) is the forcing period, and evolve under the time-periodic, piecewise-smooth vector field defined by
	\[
	f^\pm(t,x) = \begin{pmatrix} -\mu x_1 \pm \sin t \\ -\alpha x_2 \end{pmatrix}, \quad \text{with } \mu = 1.8,\ \alpha = 0.1,
	\]
and switching manifold \( \Sigma = \{x_1 = 0\} \). The figure depicts the phase-space projection of 3600 trajectories initialized on a uniform \( 60 \times 60 \) grid in \( x_1 \in [-1.5, 1.5] \), \( x_2 \in [-2, 2] \). Each polyline corresponds to a sequence of points sampled at integer multiples of \( T \), forming a discrete trajectory under the Poincaré return map. The convergence of all trajectories toward a common closed curve—regardless of their initial conditions—confirms the existence of a globally attracting limit cycle. This visual outcome supports the core result of exponential stability and uniqueness of the periodic orbit. It also emphasizes that the Poincaré map, though affected by switching and saltation events, remains a strict contraction in the induced weighted metric.

\

\begin{figure}[H]
	\centering
	\includegraphics[width=0.60\textwidth]{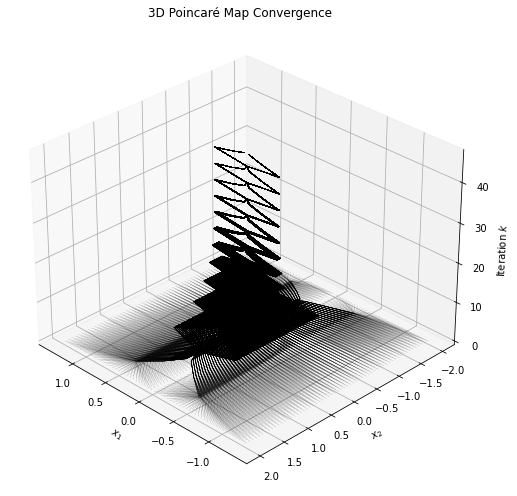}
	\caption{
		\textbf{Iterative convergence under the Poincaré map visualized in lifted 3D space.}
		This figure presents the same dataset as Figure~\ref{fig:poincare2d}, but lifted into three-dimensional space \((x_1, x_2, k)\), where \(k \in \mathbb{N}_0\) denotes the Poincaré iteration index.
		Each trajectory is shown as a polyline tracing the evolution of a Filippov solution sampled at integer multiples of the forcing period \(T\), i.e., \(x_k := \phi(kT, x_0)\), for \(k = 0, \dots, 47\).
		The layered structure highlights exponential convergence of all iterates toward the unique periodic orbit, with the vertical axis \(k\) visually encoding contraction over successive applications of the time-\(T\) map.
		This geometric contraction is a direct manifestation of the Banach fixed-point property of the Poincaré map \(\phi_T\), rigorously established in Theorem~\ref{thm:existence_stability}.
	}
	\label{fig:poincare3d}
\end{figure}

\

Figure~\ref{fig:poincare3d} provides a lifted geometric perspective on the convergence behavior of the Filippov system under the Poincaré map. Each polyline represents the discrete trajectory \( x_k = \phi(kT, x_0) \), where \( T = 2\pi \) is the forcing period and \( k \in \mathbb{N}_0 \) is the Poincaré iteration index. The plot embeds these iterates into three-dimensional space \( (x_1, x_2, k) \), where the vertical axis encodes iteration depth. This visualization complements Figure~\ref{fig:poincare2d} by revealing the vertical stratification of contraction: trajectories initialized across a broad range of initial conditions collapse onto a common limit set over increasing \( k \), demonstrating exponential convergence toward the unique periodic orbit. The narrowing funnel shape across layers indicates the decreasing separation between iterates, a direct visual signature of the contraction property of the time-\(T\) map. This layered convergence pattern confirms the Banach fixed-point principle applied to the Poincaré map \( \phi_T \), as invoked in Theorem~\ref{thm:existence_stability}. The 3D lifting makes explicit the dynamical role of iteration count and provides a clear illustration of the orbit’s global attractivity under repeated applications of the nonsmooth, time-periodic flow.

\section{Conclusion}
\label{sec:conclusion}

We have developed a contraction-based framework for analyzing time-periodic behavior in planar nonsmooth dynamical systems governed by Filippov differential inclusions. The method employs a time- and state-dependent weighted Riemannian metric, incorporates Clarke’s generalized Jacobian, and enforces a jump condition across switching manifolds to derive sufficient conditions for uniform exponential contraction on compact forward-invariant sets. These local and global estimates culminate in a Banach fixed-point argument establishing the existence and exponential stability of a unique periodic orbit.

This work generalizes existing scalar contraction approaches~\cite{giesl2005} to a broad class of two-dimensional systems with discontinuities, including sliding modes and switching phenomena. The resulting framework provides a systematic analytic pathway for studying stability in nonsmooth systems without requiring smoothness or convexity assumptions.

Several extensions remain open. These include generalizing the theory to higher-dimensional systems, formulating converse results that provide necessary as well as sufficient conditions for contraction, and constructing global Lyapunov functions in nonsmooth settings. Potential applications span hybrid control systems, nonsmooth mechanical models, and computational approaches to nonsmooth dynamics.

\section{Appendix}
\label{secx:appendix}

\begin{lemma}[Local-to-Global Contraction Implication]\label{lem:local_to_global_contraction}
	Let \( K \subset \mathbb{R}^n \) be a compact, path-connected, forward-invariant set. Suppose the Filippov flow \( \phi(t,x_0) \) satisfies the following:
	
	\begin{enumerate}
		\item[(i)] \textbf{(Local contraction)} There exist constants \( \delta > 0 \), \( \nu > 0 \), and \( C > 0 \) such that for all \( x_0, y_0 \in K \) with \( \|x_0 - y_0\| < \delta \), the inequality
		\[
		\|\phi(t, x_0) - \phi(t, y_0)\| \leq C e^{-\nu t} \|x_0 - y_0\|
		\]
		holds for all \( t \geq 0 \).
	\end{enumerate}
	
	Then there exists a constant \( C' > 0 \) such that for all \( x_0, y_0 \in K \) and all \( t \geq 0 \),
	\[
	\|\phi(t, x_0) - \phi(t, y_0)\| \leq C' e^{-\nu t} \|x_0 - y_0\|.
	\]
\end{lemma}

\begin{proof}
	Let \( x_0, y_0 \in K \) be arbitrary. Since \( K \) is compact and path-connected, there exists a Lipschitz continuous path \( \gamma : [0,1] \to K \) such that \( \gamma(0) = x_0 \) and \( \gamma(1) = y_0 \). By compactness of \( [0,1] \) and continuity of \( \gamma \), we can select a partition
	\[
	0 = s_0 < s_1 < \dots < s_m = 1
	\]
	such that for each \( j = 0, \dots, m-1 \),
	\[
	\| \gamma(s_{j+1}) - \gamma(s_j) \| < \delta.
	\]
	Define \( z_j := \gamma(s_j) \). By the local contraction assumption,
	\[
	\|\phi(t, z_{j+1}) - \phi(t, z_j)\| \leq C e^{-\nu t} \|z_{j+1} - z_j\|.
	\]
	Summing over segments and applying the triangle inequality:
	\[
	\|\phi(t, x_0) - \phi(t, y_0)\| \leq \sum_{j=0}^{m-1} \|\phi(t, z_{j+1}) - \phi(t, z_j)\| \leq C e^{-\nu t} \sum_{j=0}^{m-1} \|z_{j+1} - z_j\|.
	\]
	Since \( \gamma \) is Lipschitz and defined on a compact interval, its total variation is bounded. That is, there exists a constant \( \Lambda > 0 \) such that
	\[
	\sum_{j=0}^{m-1} \|z_{j+1} - z_j\| \leq \Lambda \|x_0 - y_0\|.
	\]
	Hence,
	\[
	\|\phi(t, x_0) - \phi(t, y_0)\| \leq C \Lambda e^{-\nu t} \|x_0 - y_0\|.
	\]
	Setting \( C' := C \Lambda \) completes the proof.
\end{proof}

\

\begin{corollary}[Exponential Decay in Euclidean Norm]\label{cor:euclidean_decay}
	Let \( K \subset \mathbb{R}^n \) be compact and forward-invariant, and suppose the weight function \( W(t,x) \) is continuous on \( K \). Define
	\[
	M := \sup_{(t,x) \in K} |W(t,x)| < \infty.
	\]
	Assume \( x(t), y(t) \) are Filippov solutions with initial conditions in \( K \setminus \Sigma \), and that the weighted distance
	\[
	A(t) := e^{W(t,x(t))} \|x(t) - y(t)\|
	\]
	satisfies the exponential contraction estimate
	\[
	A(t) \leq A(0) e^{-\nu t} \quad \text{for all } t \geq 0.
	\]
	Then the Euclidean distance also decays exponentially:
	\[
	\|x(t) - y(t)\| \leq e^{2M} \|x(0) - y(0)\| e^{-\nu t}.
	\]
\end{corollary}

\

\begin{proof}
	Since \( W(t,x) \) is continuous on the compact set \( K \), it is uniformly bounded:
	\[
	|W(t,x)| \leq M \quad \text{for all } (t,x) \in K.
	\]
	Hence,
	\[
	e^{-M} \|x(t) - y(t)\| \leq A(t) \leq A(0) e^{-\nu t}.
	\]
	At \( t = 0 \), we have
	\[
	A(0) = e^{W(0,x(0))} \|x(0) - y(0)\| \leq e^{M} \|x(0) - y(0)\|.
	\]
	Combining both bounds yields
	\[
	\|x(t) - y(t)\| \leq e^{2M} \|x(0) - y(0)\| e^{-\nu t}. \qedhere
	\]
\end{proof}

\

\begin{remark}[Contraction Results]\label{rem:contraction_results}
Lemma~\ref{lem:local_to_global_contraction} formalizes the idea that local exponential contraction of the flow—when uniform over a compact, forward-invariant domain—extends to a global contraction property via interpolation along continuous paths. This permits the derivation of global stability results from local differential inequalities, without requiring global monotonicity or convexity.
	
Corollary~\ref{cor:euclidean_decay} shows how contraction in a weighted (non-Euclidean) metric—typically arising from a time-varying Lyapunov-like function \( W(t,x) \)—implies exponential decay of the standard Euclidean distance. The result explicitly accounts for metric distortion, ensuring robustness under multiplicative rescaling.
	
This principle naturally generalizes to contraction in time-varying Riemannian metrics. If the contraction estimate holds with respect to a Riemannian metric induced by a uniformly bounded positive-definite matrix field \( M(t,x) \), then exponential decay in the Euclidean norm still follows, with constants determined by the condition number of \( M \). This highlights the utility of coordinate-invariant geometric techniques in nonsmooth stability analysis.
\end{remark}

\end{document}